\documentclass[11pt]{amsart}

\newcommand{\nc}{\newcommand}

\nc{\Ls}{{L^\star}}
\nc{\myx}{{\mathtt x}}
\nc{\myy}{{\mathtt{y}}}
\nc{\pr}{{\rm pr}}
\nc{\bfeta}{{\boldsymbol \eta}}
\nc{\lp}{{\Big(}}
\nc{\rp}{{\Big)}}
\nc{\Lp}{{\bigg(}}
\nc{\Rp}{{\bigg)}}
\def\lrp#1#2{\lp \genfrac{}{}{0pt}{}{#1}{#2} \rp}
\def\Lrp#1#2{\Lp \genfrac{}{}{0pt}{}{#1}{#2} \Rp}
\def\lrpT#1#2{\lp \genfrac{}{}{0pt}{}{#1}{#2}\Big| T \rp}

\def\lrpTZ#1#2{\lp \genfrac{}{}{0pt}{}{#1}{#2}\Big| 0 \rp}
\def\LrpTZ#1#2{\Lp \genfrac{}{}{0pt}{}{#1}{#2}\bigg| 0 \Rp}

\usepackage[enableskew]{youngtab}
\def\KeyWord#1{$\backslash$\IfColor{$\!\!$\textRed{#1}\textBlack}{#1}$\!\!$}

\usepackage{amsfonts}
\usepackage{amssymb}
\usepackage{amsmath}
\usepackage{amsthm}
\usepackage{enumerate}
\usepackage{hyperref}

\usepackage{etex}
\reserveinserts{28}
\usepackage{amsgen}
\usepackage{amsopn}
\usepackage{verbatim}
\usepackage{xspace}
\usepackage{multicol}
\usepackage{eepic}
\usepackage{url}
\usepackage{upref}
\usepackage{pgf}
\usepackage{tikz}
\usetikzlibrary{patterns}
\usepackage{pgffor}
\usepackage{pgfcalendar}
\usepackage{pgfpages}
\usepackage{shuffle,yfonts}
\usepackage{mathtools}
\DeclareFontFamily{U}{shuffle}{}
\DeclareFontShape{U}{shuffle}{m}{n}{ <-8>shuffle7 <8->shuffle10}{}

\theoremstyle{plain}
\newtheorem{thm}{Theorem}[section]
\newtheorem{lem}[thm]{Lemma}
\newtheorem{prop}[thm]{Proposition}
\newtheorem{lem-defn}[thm]{Lemma-Definition}

\newtheorem{conj}[thm]{Conjecture}

\theoremstyle{definition}

\newtheorem{rem}[thm]{Remark}

\newtheorem{eg}[thm]{Example}

\DeclareMathOperator{\Li}{{Li}}

\DeclareMathOperator{\dep}{{dp}}

\DeclareMathOperator{\ord}{{ord}}

\nc{\setA}{{\mathsf A}}
\nc{\setB}{{\mathsf B}}
\nc{\setC}{{\mathsf C}}
\nc{\setD}{{\mathsf D}}
\nc{\setE}{{\mathsf E}}
\nc{\setF}{{\mathsf F}}
\nc{\setG}{{\mathsf G}}
\nc{\setH}{{\mathsf{H}}}
\nc{\setI}{{\mathsf I}}
\nc{\setJ}{{\mathsf J}}
\nc{\setK}{{\mathsf K}}
\nc{\setL}{{\mathsf L}}
\nc{\setM}{{\mathsf M}}
\nc{\setN}{{\mathsf N}}
\nc{\setO}{{\mathsf O}}
\nc{\setP}{{\mathsf P}}
\nc{\setQ}{{\mathsf Q}}
\nc{\setR}{{\mathsf R}}
\nc{\setS}{{\mathsf S}}
\nc{\setT}{{\mathsf T}}
\nc{\setU}{{\mathsf U}}
\nc{\setV}{{\mathsf V}}
\nc{\setW}{{\mathsf W}}
\nc{\setX}{{\mathsf X}}
\nc{\setY}{{\mathsf Y}}
\nc{\setZ}{{\mathsf Z}}

\nc{\z}{{\texttt{z}}}

\nc{\exref}[1]{\chapref\ref{#1}}

\nc{\per}[1]{\underset{#1}{\boldsymbol \pi}\,}
\nc{\QSym}{{\mathsf{QSym}}}
\nc{\tla}{\overset{\leftarrow}}
\nc{\MT}{{\rm MT}}
\nc{\XX}{{X}}
\nc{\gF}{{\varPhi}}
\nc{\gL}{{\Lambda}}
\nc{\ot}{\otimes}

\nc{\dual}{\ast}
\nc{\wht}{\widehat}
\nc{\bwg}{{\bigwedge}}
\nc{\wg}{{\wedge}}
\nc{\mmu}{{\boldsymbol{\mu}}}
\nc{\mal}{{{\scstl \maltese}}}
\nc{\fA}{{\mathfrak A}}
\nc{\hfA}{{\widehat{\mathfrak A}}}
\nc{\HH}{{\mathbb H}}
\nc{\ra}{\rightarrow}
\nc{\os}{{\overset}}
\nc{\GG}{{\mathbb G}}
\nc{\F}{{\mathbb F}}
\nc{\Z}{{\mathbb Z}}
\nc{\R}{{\mathbb R}}
\nc{\N}{{\mathbb N}}
\nc{\ZN}{{{\mathbb N}_0}}
\nc{\Q}{{\mathbb Q}}
\nc{\CC}{{\mathbb C}}
\nc{\V}{{\mathbb V}}
\nc{\CP}{{\mathbb{CP}}}
\nc{\Cnn}{{\mathbb C}_{\ge 0}}
\nc{\Cp}{{\mathbb C}_{>0}}

\nc{\MHSS}{MH${}^\star$S\xspace}
\nc{\MHSSs}{MH${}^\star$Ss\xspace}
\nc{\MZSV}{MZ${}^\star$V\xspace}
\nc{\MZSVs}{MZ${}^\star$Vs\xspace}
\nc{\FMZSV}{FMZ${}^\star$V\xspace}
\nc{\FMZSVs}{FMZ${}^\star$Vs\xspace}
\nc{\FESS}{FE${}^\star$S\xspace}
\nc{\FESSs}{FE${}^\star$Ss\xspace}

\nc{\DSh}{{\mathsf{DSh}}}
\nc{\ShC}{{\mathsf{ShC}}}
\nc{\MZV}{{\mathsf{MZ}}}
\nc{\SCV}{{\mathsf{SCV}}}
\nc{\nSCV}{{\widetilde{\mathsf{SCV}}}}
\nc{\FCV}{{\mathsf{FCV}}}
\nc{\qMZ}{{q\mathsf{MZ}}}
\nc{\FMZ}{{\mathsf{FMZ}}}
\nc{\CMZV}{{\mathsf{CMZV}}}
\nc{\ES}{{\mathsf{ES}}}
\nc{\FES}{{\mathsf{FES}}}
\nc{\qMZV}{\mathsf{qMZV}}
\nc{\grqMZV}{\mathsf{grqMZV}}

\nc{\gemn}{{\mathfrak n}}

\nc{\wtcalM}{{\widetilde\calM}}
\nc{\gam}{{\gamma}}
\nc{\gG}{{\Gamma}}
\nc{\om}{{\omega}}
\nc{\vep}{{\varepsilon}}
\nc{\ga}{{\alpha}}
\nc{\gl}{{\lambda}}
\nc{\gb}{{\beta}}
\nc{\gd}{{\delta}}
\nc{\gf}{{\varphi}}
\nc{\orgd}{{\vec \gd\,}}
\nc{\gs}{{\sigma}}
\nc{\gth}{{\theta}}
\nc{\gS}{{\Sigma}}

\nc{\gk}{{\kappa}}

\nc{\tgz}{{\tilde{\zeta}}}
\nc{\gO}{{\Omega}}
\nc{\sif}{{\mathcal S}}
\nc{\gt}{{\tau}}
\nc{\Lra}{\Longrightarrow}
\nc{\lra}{\longrightarrow}
\nc{\lmaps}{\longmapsto}
\nc{\fS}{{\mathsf S}}
\nc{\DD}{{\mathfrak D}}
\nc{\Llra}{\Longleftrightarrow}
\nc{\ol}{\overline}
\def\ola#1{\overset{\text{\raisebox{-4pt}{$\scriptscriptstyle \leftarrow$}}}{#1}}

\nc{\zq}{{\zeta_q}}
\nc{\us}{\underset}
\nc{\tn}{{\tilde{n}}}
\nc{\gD}{{\Delta}}
\nc{\bi}{{\bf i}}
\nc{\bfone}{{\bf 1}}
\nc{\bfzero}{{\bf 0}}
\nc{\bftwo}{{\bf 2}}

\nc{\bfa}{{\bf a}}
\nc{\bfb}{{\bf b}}
\nc{\bfc}{{\bf c}}
\nc{\bfd}{{\bf d}}
\nc{\bfe}{{\bf e}}
\nc{\bferev}{\ola{\bf e}}
\nc{\bff}{{\bf f}}
\nc{\bfg}{{\bf g}}
\nc{\bfh}{{\bf h}}
\nc{\bfi}{{\bf i}}
\nc{\bfj}{{\bf j}}
\nc{\bfjrev}{\ola{\bf j}}
\nc{\obfi}{{\overrightarrow{\boldsymbol \imath}}}
\nc{\obfj}{{\overrightarrow{\boldsymbol \jmath}}}
\nc{\obfk}{{\overrightarrow{\bf k}}}
\nc{\veps}{{\varepsilon}}
\nc{\bfn}{{\bf n}}
\nc{\bfl}{{\bf l}}
\nc{\bfk}{{\bf k}}
\nc{\bfm}{{\bf m}}
\nc{\bfo}{{\bf o}}
\nc{\bfp}{{\bf p}}
\nc{\bfq}{{\bf q}}
\nc{\bfr}{{\bf r}}
\nc{\bfs}{{\bf s}}
\nc{\bft}{{\bf t}}
\nc{\bfu}{{\bf u}}
\nc{\bfv}{{\bf v}}
\nc{\bfw}{{\bf w}}
\nc{\bfx}{{\bf x}}
\nc{\bfy}{{\bf y}}
\nc{\bfz}{{\bf z}}
\nc{\bfA}{{\bf A}}
\nc{\bfB}{{\bf B}}
\nc{\bfC}{{\bf C}}
\nc{\bfD}{{\bf D}}
\nc{\bfE}{{\bf E}}
\nc{\bfF}{{\bf F}}
\nc{\bfG}{{\bf G}}
\nc{\bfH}{{\bf H}}
\nc{\bfI}{{\bf I}}
\nc{\bfJ}{{\bf J}}
\nc{\bfK}{{\bf K}}
\nc{\bfL}{{\bf L}}
\nc{\bfM}{{\bf M}}
\nc{\bfN}{{\bf N}}
\nc{\bfO}{{\bf O}}
\nc{\bfP}{{\bf P}}
\nc{\bfQ}{{\bf Q}}
\nc{\bfR}{{\bf R}}
\nc{\bfS}{{\bf S}}
\nc{\bfT}{{\bf T}}
\nc{\bfU}{{\bf U}}
\nc{\bfV}{{\bf V}}
\nc{\bfW}{{\bf W}}
\nc{\bfX}{{\bf X}}
\nc{\bfY}{{\bf Y}}
\nc{\bfZ}{{\bf Z}}
\nc{\bbA}{{\mathbb A}}
\nc{\bbB}{{\mathbb B}}
\nc{\bbC}{{\mathbb C}}
\nc{\bbD}{{\mathbb D}}
\nc{\bbE}{{\mathbb E}}
\nc{\bbF}{{\mathbb F}}
\nc{\bbG}{{\mathbb G}}
\nc{\bbH}{{\mathbb H}}
\nc{\bbI}{{\mathbb I}}
\nc{\bbJ}{{\mathbb J}}
\nc{\bbK}{{\mathbb K}}
\nc{\bbL}{{\mathbb L}}
\nc{\bbM}{{\mathbb M}}
\nc{\bbN}{{\mathbb N}}
\nc{\bbO}{{\mathbb O}}
\nc{\bbP}{{\mathbb P}}
\nc{\bbQ}{{\mathbb Q}}
\nc{\bbR}{{\mathbb R}}
\nc{\bbS}{{\mathbb S}}
\nc{\bbT}{{\mathbb T}}
\nc{\bbU}{{\mathbb U}}
\nc{\bbV}{{\mathbb V}}
\nc{\bbW}{{\mathbb W}}
\nc{\bbX}{{\mathbb X}}
\nc{\bbY}{{\mathbb Y}}
\nc{\bbZ}{{\mathbb Z}}
\nc{\bba}{{\mathbb a}}
\nc{\bbb}{{\mathbb b}}
\nc{\bbc}{{\mathbb c}}
\nc{\bbd}{{\mathbb d}}
\nc{\bbe}{{\mathbb e}}
\nc{\bbf}{{\mathbb f}}
\nc{\bbg}{{\mathbb g}}
\nc{\bbh}{{\mathbb h}}
\nc{\bbi}{{\mathbb i}}
\nc{\bbk}{{\mathbb K}}
\nc{\bbl}{{\mathbb l}}
\nc{\bbm}{{\mathbb m}}
\nc{\bbn}{{\mathbb n}}
\nc{\bbo}{{\mathbb o}}
\nc{\bbp}{{\mathbb p}}
\nc{\bbq}{{\mathbb q}}
\nc{\bbr}{{\mathbb r}}
\nc{\bbs}{{\mathbb s}}
\nc{\bbt}{{\mathbb t}}
\nc{\bbu}{{\mathbb u}}
\nc{\bbv}{{\mathbb v}}
\nc{\bbw}{{\mathbb w}}
\nc{\bbx}{{\mathbb x}}
\nc{\bby}{{\mathbb y}}
\nc{\bbz}{{\mathbb z}}

\nc{\calA}{{\mathcal A}}
\nc{\calB}{{\mathcal B}}
\nc{\calC}{{\mathcal C}}
\nc{\calD}{{\mathcal D}}
\nc{\calE}{{\mathcal E}}
\nc{\calF}{{\mathcal F}}
\nc{\calG}{{\mathcal G}}
\nc{\calH}{{\mathcal H}}
\nc{\calI}{{\mathcal I}}
\nc{\calJ}{{\mathcal J}}
\nc{\calK}{{\mathcal K}}
\nc{\calL}{{\mathcal L}}
\nc{\calM}{{\mathcal M}}
\nc{\calN}{{\mathcal N}}
\nc{\calO}{{\mathcal O}}
\nc{\calP}{{\mathcal P}}
\nc{\calQ}{{\mathcal Q}}
\nc{\calR}{{\mathcal R}}
\nc{\calS}{{\mathcal S}}
\nc{\calT}{{\mathcal T}}
\nc{\calU}{{\mathcal U}}
\nc{\calV}{{\mathcal V}}
\nc{\calW}{{\mathcal W}}
\nc{\calX}{{\mathcal X}}
\nc{\calY}{{\mathcal Y}}
\nc{\calZ}{{\mathcal Z}}

\usepackage[cal=boondoxo]{mathalfa}

 \nc{\cala}{{\mathcal a}}
\nc{\calb}{{\mathcal b}}
\nc{\calc}{{\mathcal c}}
\nc{\cald}{{\mathcal d}}
\nc{\cale}{{\mathcal e}}
\nc{\calf}{{\mathcal f}}
\nc{\calg}{{\mathcal g}}
\nc{\calh}{{\mathcal h}}
\nc{\cali}{{\mathcal i}}
\nc{\calj}{{\mathcal j}}
\nc{\calk}{{\mathcal k}}
\nc{\call}{{\mathcal l}}
\nc{\calm}{{\mathcal m}}
\nc{\caln}{{\mathcal n}}
\nc{\calo}{{\mathcal o}}
\nc{\calp}{{\mathcal p}}
\nc{\calq}{{\mathcal q}}
\nc{\calr}{{\mathcal r}}
\nc{\cals}{{\mathcal s}}
\nc{\calt}{{\mathcal t}}
\nc{\calu}{{\mathcal u}}
\nc{\calv}{{\mathcal v}}
\nc{\calw}{{\mathcal w}}
\nc{\calx}{{\mathcal x}}
\nc{\caly}{{\mathcal y}}
\nc{\calz}{{\mathcal z}}

\nc{\frakA}{{\mathfrak A}}
\nc{\frakB}{{\mathfrak B}}
\nc{\frakC}{{\mathfrak C}}
\nc{\frakD}{{\mathfrak D}}
\nc{\frakE}{{\mathfrak E}}
\nc{\frakF}{{\mathfrak F}}
\nc{\frakG}{{\mathfrak G}}
\nc{\frakH}{{\mathfrak H}}
\nc{\frakI}{{\mathfrak I}}
\nc{\frakJ}{{\mathfrak J}}
\nc{\frakK}{{\mathfrak K}}
\nc{\frakL}{{\mathfrak L}}
\nc{\frakM}{{\mathfrak M}}
\nc{\frakN}{{\mathfrak N}}
\nc{\frakO}{{\mathfrak O}}
\nc{\frakP}{{\mathfrak P}}
\nc{\frakQ}{{\mathfrak Q}}
\nc{\frakR}{{\mathfrak R}}
\nc{\frakS}{{\mathfrak S}}
\nc{\frakT}{{\mathfrak T}}
\nc{\frakU}{{\mathfrak U}}
\nc{\frakV}{{\mathfrak V}}
\nc{\frakW}{{\mathfrak W}}
\nc{\frakX}{{\mathfrak X}}
\nc{\frakY}{{\mathfrak Y}}
\nc{\frakZ}{{\mathfrak Z}}
\nc{\fraka}{{\mathfrak a}}
\nc{\frakb}{{\mathfrak b}}
\nc{\frakc}{{\mathfrak c}}
\nc{\frakd}{{\mathfrak d}}
\nc{\frake}{{\mathfrak e}}
\nc{\frakf}{{\mathfrak f}}
\nc{\frakg}{{\mathfrak g}}
\nc{\frakh}{{\mathfrak h}}
\nc{\fraki}{{\mathfrak i}}
\nc{\frakj}{{\mathfrak j}}
\nc{\frakk}{{\mathfrak k}}
\nc{\frakl}{{\mathfrak l}}
\nc{\frakm}{{\mathfrak m}}
\nc{\frakn}{{\mathfrak n}}
\nc{\frako}{{\mathfrak o}}
\nc{\frakp}{{\mathfrak p}}
\nc{\frakq}{{\mathfrak q}}
\nc{\frakr}{{\mathfrak r}}
\nc{\fraks}{{\mathfrak s}}
\nc{\frakt}{{\mathfrak t}}
\nc{\fraku}{{\mathfrak u}}
\nc{\frakv}{{\mathfrak v}}
\nc{\frakw}{{\mathfrak w}}
\nc{\frakx}{{\mathfrak x}}
\nc{\fraky}{{\mathfrak y}}
\nc{\frakz}{{\mathfrak z}}
\nc{\so}{{\mathfrak so}}
\nc{\slfour}{{\mathfrak sl}_4}
\nc{\one}{{\bf 1}}
\nc{\zero}{{\bf 0}}

\nc{\sha}{\shuffle}
\nc{\inj}{\hookrightarrow}

\nc{\zetas}{{\zeta^\star}}

\nc{\DLN}{{\mathsf{DivLog}_\N}}
\nc{\DLD}{{\mathsf{DivLog}_D}}

\nc{\invdots}{{.\text{\raisebox{0.2ex}{$\cdot$}} \text{\raisebox{0.9ex}{$\cdot$}} }}
\nc{\Sy}{{\mathcal S}}
\nc{\inv}{{\rm inv}}
\nc{\Rac}{{\mathcal R}}
\nc{\dd}{{\mathfrak d}}
\nc{\shD}{{\mbox{\cyr D}}}
\nc{\tz}{\tilde{\zeta}}

\nc{\x}{{\mathtt{x}}}
\nc{\y}{{\mathtt{y}}}
\nc{\Qxy}{\Q\langle \x,\y\rangle}

\nc{\recurO}{{{\mathfrak D}_{z,\barz}}}
\nc{\recurI}{{{\mathfrak I}_{z,\barz}}}
\nc{\app}{{\sharp}}
\nc{\barz}{{\bar{z}}}
\nc{\bB}{{\mathsf{B}}}
\nc{\angX}{{\langle\!\langle \setX\rangle\!\rangle}}
\nc{\revs}{\ola}
\nc{\res}{{\rm{res}}}
\nc{\sv}{{\rm{sv}}}
\nc{\tdx}{{\y}}
\nc{\setLs}{{\mathsf{Ls}}}

\nc{\eps}{{\epsilon}}

\nc{\lsetH}{{\hat{\setH}}}
\nc{\Sooo}{\mathsf{S}^0_{0,0}\,}
\nc{\Solo}{\mathsf{S}^0_{1,0}\,}
\nc{\Sool}{\mathsf{S}^0_{0,1}\,}
\nc{\Soll}{\mathsf{S}^0_{1,1}\,}
\nc{\Sllo}{\mathsf{S}^1_{1,0}\,}
\nc{\Slol}{\mathsf{S}^1_{0,1}\,}
\nc{\Sloo}{\mathsf{S}^1_{0,0}\,}
\nc{\Slll}{\mathsf{S}^1_{1,1}\,}
\nc{\Stot}{\mathsf{S}_\setX}
\nc{\Stotnx}{\mathsf{S}}
\nc{\Zooo}{\Phi^0_{0,0}\,}
\nc{\Zolo}{\Phi^0_{1,0}\,}
\nc{\Zool}{\Phi^0_{0,1}\,}
\nc{\Zloo}{\Phi^1_{0,0}\,}
\nc{\Zllo}{\Phi^1_{1,0}\,}
\nc{\Zlol}{\Phi^1_{0,1}\,}
\nc{\Zoll}{\Phi^0_{1,1}\,}
\nc{\vooo}{V^0_{0,0}\,}
\nc{\volo}{V^0_{1,0}\,}
\nc{\vool}{V^0_{0,1}\,}
\nc{\vloo}{V^1_{0,0}\,}
\nc{\vllo}{V^1_{1,0}\,}
\nc{\vlol}{V^1_{0,1}\,}
\nc{\voll}{V^0_{1,1}\,}
\nc{\vlll}{V^1_{1,1}\,}
\nc{\Fo}{F}

\nc{\lSooo}{\hat{\mathsf{S}}^0_{0,0}\,}
\nc{\lSolo}{\hat{\mathsf{S}}^0_{1,0}\,}
\nc{\lSool}{\hat{\mathsf{S}}^0_{0,1}\,}
\nc{\lSoll}{\hat{\mathsf{S}}^0_{1,1}\,}
\nc{\lSllo}{\hat{\mathsf{S}}^1_{1,0}\,}
\nc{\lSlol}{\hat{\mathsf{S}}^1_{0,1}\,}
\nc{\lSloo}{\hat{\mathsf{S}}^1_{0,0}\,}
\nc{\lSlll}{\hat{\mathsf{S}}^1_{1,1}\,}
\nc{\lStot}{{\hat{\mathsf{S}}_\setX}}
\nc{\lStotnx}{\hat{\mathsf{S}}}
\nc{\lZooo}{\hat{\Phi}^0_{0,0}\,}
\nc{\lZolo}{\hat{\Phi}^0_{1,0}\,}
\nc{\lZool}{\hat{\Phi}^0_{0,1}\,}
\nc{\lZloo}{\hat{\Phi}^1_{0,0}\,}
\nc{\lZllo}{\hat{\Phi}^1_{1,0}\,}
\nc{\lZlol}{\hat{\Phi}^1_{0,1}\,}
\nc{\lvooo}{\hat{V}^0_{0,0}\,}
\nc{\lvolo}{\hat{V}^0_{1,0}\,}
\nc{\lvool}{\hat{V}^0_{0,1}\,}
\nc{\lvloo}{\hat{V}^1_{0,0}\,}
\nc{\lvllo}{\hat{V}^1_{1,0}\,}
\nc{\lvlol}{\hat{V}^1_{0,1}\,}
\nc{\lvoll}{\hat{V}^0_{1,1}\,}
\nc{\lvlll}{\hat{V}^1_{1,1}\,}
\nc{\lZo}{\hat{\Phi}_0}
\nc{\lZp}{\hat{\Phi}_{\pi}}
\nc{\lZH}{\hat{\Phi}_{H}}
\nc{\lZs}{\hat{\Phi}_{s}}
\nc{\lF}{\hat{F}}
\nc{\lH}{{\hat{H}}}

\nc{\gFF}{{\chi}}
\nc{\gX}{{\varPsi}}
\nc{\gXs}{\gX^\star}
\nc{\cv}{{\rm cv}}
\nc{\myone}{{1}}
\nc{\whH}{\widehat{H}}
\nc{\whR}{{R}}

\nc{\KZ}{{\rm KZ}}
\nc{\even}{{\rm ev}}
\nc{\odd}{{\rm od}}

\nc{\na}{\natural}
\nc{\tT}{\tilde{T}}

\nc{\db}{{\mathbb D}}

\nc{\BTF}{{F}}
\nc{\gbb}{{\beta}}

\nc{\bfzt}{{{\boldsymbol \zeta}}}
\nc{\tf}{{\tilde{f}}}

\nc{\sHq}{{\mathfrak h}}
\nc{\tHq}{{\tilde{H}}}
\nc{\stHq}{{\tilde{\sHq}}}
\nc{\hHq}{{\hat{H}}}
\nc{\shHq}{{\hat{\sHq}}}

\nc{\bigcc}{\operatorname{\phantom{a} \text{\raisebox{1pt}{$\scstl \circ$}}\hskip-2.2ex\bigsqcup}}

\nc{\hone}{{\widehat{1}}}
\nc{\genf}{\genfrac{[}{]}{0pt}{}}

\nc{\oll}[1]{\underline{#1}}
\nc{\hyf}{{\text{-}}}
\nc{\bt}{{\bf 2}}
\nc{\vbfp}{{\underline{\bfp}}}
\nc{\vbfs}{{\underline{\bfs}}}
\nc{\vbfm}{{\underline{\bfm}}}
\nc{\vbfw}{{\underline{\bfw}}}

\nc{\wbfp}{{\widetilde{\bfp}}}
\nc{\wbfs}{{\widetilde{\bfs}}}
\nc{\wbfw}{{\widetilde{\bfw}}}
\nc{\wbfm}{{\widetilde{\bfm}}}
\nc{\wdt}{{\widetilde{t}}}
\nc{\wdl}{{\widetilde{\gl}}}
\nc{\wdp}{{\widetilde{p}}}
\nc{\wwbfw}{{\overset{\text{\raisebox{-2pt}{$\approx$}}}{\bfw}}}
\nc{\wwbfp}{{\overset{\text{\raisebox{-2pt}{$\approx$}}}{\bfp}}}
\nc{\wwbfs}{{\overset{\text{\raisebox{-2pt}{$\approx$}}}{\bfs}}}
\nc{\wwbfm}{{\overset{\text{\raisebox{-2pt}{$\approx$}}}{\bfm}}}
\nc{\wwdp}{{\overset{\text{\raisebox{-2pt}{$\approx$}}}{\!p}}}
\nc{\wwdl}{{\overset{\text{\raisebox{-2pt}{$\approx$}}}{\!\gl}}}
\nc{\tb}{{\tilde{b}}}
\nc{\tB}{{\widetilde{B}}}
\nc{\tX}{{\widetilde{X}}}
\nc{\tY}{{\widetilde{Y}}}
\nc{\tbfs}{{\tilde{\bfs}}}
\nc{\tbft}{{\tilde{\bft}}}
\nc{\tbfu}{{\tilde{\bfu}}}
\nc{\ttbfs}{{\hat{\bfs}}}
\nc{\ttbft}{{\hat{\bft}}}
\nc{\ttbfu}{{\hat{\bfu}}}
\nc{\rrho}{{\hat{\rho}}}
\nc{\ggk}{{\hat{\gk}}}
\nc{\ggs}{{\hat{\gs}}}
\nc{\oI}{{\overline{I}}}
\nc{\bI}{{\bar{I}}}
\nc{\bJ}{{\bar{J}}}
\nc{\bK}{{\bar{K}}}

\nc{\bfgb}{{\boldsymbol \gb}}
\nc{\bfgl}{{\boldsymbol \gl}}
\nc{\wbfgl}{{\widetilde{\bfgl}}}
\nc{\wwbfgl}{{\overset{\text{\raisebox{-2pt}{$\approx$}}}{\bfgl}}}
\nc{\bfga}{{\boldsymbol \ga}}
\nc{\bfgs}{{\boldsymbol \gs}}
\nc{\bfxi}{{\boldsymbol \xi}}
\nc{\bfmu}{{\boldsymbol \mu}}
\nc{\bfnu}{{\boldsymbol \nu}}
\nc{\bftau}{{\boldsymbol \tau}}
\nc{\bfchi}{{\boldsymbol \chi}}

\nc{\Czeta}{{\calC}}
\nc{\SC}{{S}}
\nc{\htt}{{\rm ht}}
\nc{\uar}{{\uparrow}}
\nc{\dar}{{\downarrow}}

\nc{\RC}{{R}}
\nc{\QX}{{\Q\langle \setX\rangle}}
\nc{\QY}{{\Q\langle \setY\rangle}}
\nc{\CX}{{\CC\langle \setX\rangle}}
\nc{\CY}{{\CC\langle \setY\rangle}}
\nc{\RX}{{R\langle \setX\rangle}}
\nc{\RY}{{R\langle \setY\rangle}}
\nc{\QXX}{{\Q\langle\!\langle \setX\rangle\!\rangle}}
\nc{\QYY}{{\Q\langle\!\langle \setY\rangle\!\rangle}}
\nc{\CXX}{{\CC\langle\!\langle \setX\rangle\!\rangle}}
\nc{\CYY}{{\CC\langle\!\langle \setY\rangle\!\rangle}}
\nc{\RXX}{{R\langle\!\langle \setX\rangle\!\rangle}}
\nc{\RYY}{{R\langle\!\langle \setY\rangle\!\rangle}}
\nc{\lXr}{{\langle\!\langle \setX\rangle\!\rangle}}
\nc{\lYr}{{\langle\!\langle \setY\rangle\!\rangle}}
\nc{\Cat}[2]{\mathop{\mathrm{\bf Cat}}\limits_{#1}^{#2}}
\nc{\ones}{\{1\}}
\nc{\ud}{{\rm d}}
\newcommand {\be}{\myone} 

\title{Finite and symmetrized colored multiple zeta values}

\author[J.~Singer]{Johannes Singer}
\address{Department Mathematik, Friedrich-Alexander-Universit\"at Erlangen-N\"urnberg, Cauerstra\ss e 11, 91058 Erlangen, Germany}
\email{singer@math.fau.de}
\urladdr{www.math.fau.de/singer}

\author[J.~Zhao]{Jianqiang Zhao}
\address{ICMAT, C/Nicol\'as Cabrera, no.~13-15, 28049 Madrid, Spain}
\email{zhaoj@ihes.fr}
\urladdr{}
\date{}

\begin{document}

\maketitle

\begin{abstract}
Colored multiple zeta values are special values of multiple polylogarithms evaluated at $N$th roots of unity.
In this paper, we define both the finite and the symmetrized versions of these values and
show that they both satisfy the double shuffle relations.
Further, we provide strong evidence for an isomorphism connecting the two spaces
generated by these two kinds of values.
This is a generalization of a recent work of Kaneko and Zagier on finite and symmetrized multiple zeta
values and of the second author on finite and symmetrized Euler sums.
\end{abstract}

\tableofcontents


\section{Introduction}
\label{sec:Intro}

Let $\gG_N$ be the group of $N$th roots of unity. For $\bfs:=(s_1,\dots,s_d) \in \mathbb{N}^d$ and $\bfeta:=(\eta_1,\dots,\eta_d)\in (\gG_N)^d$, we define
\emph{colored multiple zeta values} (CMZVs) by
\begin{align}\label{def:CMZV}
\zeta\lrp{\bfs}{\bfeta}:=\sum_{k_1>k_2>\cdots>k_d>0}
\frac{\eta_1^{k_1}\cdots\eta_d^{k_d}}{k_1^{s_1}\cdots k_d^{s_d}}.
\end{align}
We call $d$ the \emph{depth} and $|\bfs|:=s_1+\dots+s_d$ the \emph{weight}.
These objects were first systematically studied by Deligne, Goncharov, Racinet,
Arakawa and Kaneko \cite{Arakawa99,Deligne05,Goncharov01,Racinet02}.

It is not hard to see that the series in \eqref{def:CMZV} diverge if and only if $(s_1,\eta_1)=(1,1)$. By multiplying these series we get the so called stuffle (or quasi-shuffle) relations. Additionally, it turns out that these values can also be expressed by iterated integrals
which lead to the shuffle relations. Note that for $N=1$ we rediscover \emph{multiple zeta values} (MZVs)
\begin{align}\label{def:MZV}
 \zeta(\bfs):= \sum_{k_1>k_2>\cdots>k_d>0}
\frac{1}{k_1^{s_1}\cdots k_d^{s_d}} = \zeta\lrp{\bfs}{\{1\}^d},
\end{align}
where $\{s\}^d:=(s,\ldots,s)\in \mathbb{N}^d$. When $N=2$ the colored multiple zeta values are usually called \emph{Euler sums}.

In \cite{Ihara06} Ihara, Kaneko and Zagier defined the regularized MZVs in two different ways and then obtained the regularized double shuffle relations. Racinet \cite{Racinet02} and Arakawa and Kaneko \cite{Arakawa04} 
further generalized these regularized values to arbitrary levels, which we denote by $\displaystyle \zeta_\ast\lrpT{\bfs}{\bfeta}$ and $\displaystyle\zeta_\sha\lrpT{\bfs}{\bfeta}$, 
which are in general polynomials of $T$.  Using these values,
we define the \emph{symmetrized colored multiple zeta values} (SCVs) of level $N$ by
\begin{align}
\zeta_\ast^\Sy\lrp{\bfs}{\bfeta}:=&\, \sum_{j=0}^d
 (-1)^{s_1+\cdots+s_j} \ol{\eta_1}\cdots\ol{\eta_j}\,\zeta_\ast\lrpT{s_j,\ldots,s_1}{\ol{\eta_j},\ldots,\ol{\eta_1}} \zeta_\ast\lrpT{s_{j+1},\dots,s_d}{\eta_{j+1},\dots,\eta_d}, \label{equ:astSCV}\\
\zeta_\sha^\Sy\lrp{\bfs}{\bfeta}:=&\, \sum_{j=0}^d
 (-1)^{s_1+\dots+s_j} \ol{\eta_1}\cdots\ol{\eta_j}\,\zeta_\sha\lrpT{s_j,\dots,s_1}{\ol{\eta_j},\dots,\ol{\eta_1}} \zeta_\sha\lrpT{s_{j+1},\dots,s_d}{\eta_{j+1},\dots,\eta_d}  \label{equ:shaSCV}
\end{align}
for all $\bfs\in \mathbb{N}^d$ and $\bfeta \in (\gG_N)^d$. This definition includes as special cases the \emph{symmetrized MZVs} (when $N=1$) introduced by Kaneko and Zagier (see also Jarossay \cite{Jarossay14}) and the \emph{symmetrized Euler sums} (when $N=2$) established
in \cite{Zhao15} by the second author.

We will show that SCVs are actually independent of $T$ (Proposition \ref{prop:SCVconst}) and the two versions are essentially the same modulo $\zeta(2)$ (Theorem \ref{theo:moduloz2}). Furthermore, we prove that they satisfy both the stuffle relations (Theorem \ref{thm:astSCVmorphism}) and the shuffle relations (Theorem \ref{thm:shuffleSCV})
by using two different Hopf algebra structures, respectively.

Let $\calP$ be the set of rational primes and $\F_p$ the finite field of $p$ elements. Set
\begin{align*}
\calP(N):=\{p\in \calP \colon p\equiv -1 \pmod{N} \},
\end{align*}
which is of infinite cardinality with density $1/\varphi(N)$
by Chebotarev's Density Theorem \cite{Tschebotareff26}, where $\varphi(N)$ is
Euler's totient function. Further, we define
\begin{equation*}
    \F_p[\xi_N]:=\frac{\F_p[X]}{(X^N-1)}=
    \left\{ \sum_{j=0}^{N-1} c_j \xi_N^j\colon c_0,\dots,c_{N-1}\in \F_p\right\},
\end{equation*}
where $\xi_N:=\xi_{N,p}$ is a fixed primitive root of $X^N-1\in\F_p[X]$. Moreover, we denote
\begin{equation*}
\calA(N):=\prod_{p\in\calP(N)} \F_p[\xi_N] \bigg/\bigoplus_{p\in\calP(N)} \F_p[\xi_N].
\end{equation*}
We usually remove the dependence of $\xi_N$ on $p$ by abuse of notation. For convenience,
we also identify
$\prod_{p\in\calP(N),p>k} \F_p[\xi_N] \bigg/\bigoplus_{p\in\calP(N),p>k} \F_p[\xi_N]$
with $\calA(N)$ by setting the components $a_p=0$ for all $p\le k$.

Now we define the \emph{finite colored multiple zeta values} (FCVs) of level $N$ by
\begin{equation}\label{equ:defnFCV}
\zeta_{\calA(N)} \lrp{\bfs}{\bfeta}:=\left( \sum_{p>k_1>k_2>\cdots>k_d>0}
\frac{\eta_1^{k_1}\cdots\eta_d^{k_d}}{k_1^{s_1}\cdots k_d^{s_d}}\right)_{p\in\calP(N)} \in \calA(N)
\end{equation}
for all $\bfs\in\N^d$ and $\bfeta\in(\gG_N)^d$. Again, this definition includes as special cases the \emph{finite MZVs} (when $N=1$) and the \emph{finite Euler sums} (when $N=2$).

\begin{rem}\label{rem:subtleDefn}
In fact, we have abused the notation in \eqref{equ:defnFCV}. All $\eta_j=\eta_{j,p}$ depend on $p$ so that by a fixed choice $\bfeta\in (\gG_N)^d$ we really mean a fixed choice
of $(e_1,\dots,e_d)\in(\Z/N\Z)^d$ independent of $p$ such that $\eta_{j,p}:=\xi_{N,p}^{e_j}$
for all $p$.
\end{rem}

Similar to SCVs, we shall show that FCVs satisfy both the stuffle and the shuffle relations in Theorem \ref{thm:stuffleFCV} and Theorem \ref{thm:shuffleFCV}, respectively.

The primary motivation for this paper is the following conjectural relation between
SCVs and FCVs. Let $\CMZV_{w,N}$ (resp.\ $\FCV_{w,N}$,  resp.\ $\SCV_{w,N}$)
be the $\Q(\gG_N)$-space generated by CMZVs (resp.\ FCVs, resp.\ SCVs) of weight $w$
and level $N$. Set $\nSCV_{0,N}:=\Q(\gG_N)$ and define
$\nSCV_{w,N}:=\SCV_{w,N}+2\pi i\, \SCV_{w-1,N}$ for all $w\ge 1$.

\begin{conj}\label{conj:Main}
Let $N\ge 3$ and weight $w\ge 1$. Then:
\begin{itemize}
  \item [\upshape{(i)}] $\nSCV_{w,N}=\CMZV_{w,N}$ as vector spaces over $\Q(\xi_N)$.
  \item [\upshape{(ii)}] We have a $\Q(\gG_N)$-algebra isomorphism
\begin{align*}
    f\colon \FCV_{w,N} &\, \overset{\sim}{\lra} \frac{\CMZV_{w,N}}{2\pi i\, \CMZV_{w-1,N}} \\
     \zeta_{\calA(N)}\lrp{\bfs}{\bfeta} &\,\lmaps \ \zeta_{\sha}^\Sy\lrp{\bfs}{\bfeta}.
\end{align*}
\end{itemize}
\end{conj}
It follows from this conjecture that
\begin{align*}
\frac{\CMZV_{w,N}}{2\pi i\, \CMZV_{w-1,N}}
\cong \frac{\nSCV_{w,N}}{2\pi i\, \nSCV_{w-1,N}}
\cong \frac{\SCV_{w,N}}{\big(2\pi i\, \SCV_{w-1,N}+\zeta(2)\SCV_{w-2,N}\big)\cap \SCV_{w,N}}.
\end{align*}
So the map $f$ in Conjecture~\ref{conj:Main} (ii) is surjective. The analogous result
for MZVs at level $1$ corresponding to Conjecture~\ref{conj:Main} (i) has
been proved by Yasuda \cite{Yasuda2014}.
For numerical examples supporting Conjecture~\ref{conj:Main} at level 3 and 4, see Section \ref{sec:numex}.
This conjecture should be regarded as a generalization of the corresponding conjectures
of Kaneko and Zagier for the MZVs and of the second author for the Euler sums \cite{Zhao15},
where $2\pi i$ is replaced by $\zeta(2)$ since the MZVs and Euler sums are all real numbers.
We are sure the final proof of Conjecture~\ref{conj:Main} will need the $p$-adic version
of the generalized Drinfeld associators whose coefficients should satisfy the same algebraic
relations as those of CMZVs plus the equation ``$2\pi i=0$'' when level $N\ge 3$.

\bigskip
{\bf{Acknowledgements.}}
The authors would like to thank the ICMAT at Madrid, Spain, for its warm hospitality and
gratefully acknowledge the support by the Severo Ochoa Excellence Program.

\bigskip


\section{Algebraic framework}
\label{sec:algframe}
The study of MZVs with  word algebras was initiated by Hoffman \cite{Hoffman97}
and generalized  by Racinet \cite{Racinet02} to deal with colored MZVs, which
we now review briefly.

Fix a positive integer $N$ as the level. Define the alphabet $X_N:=\{x_\eta\colon \eta \in \gG_N\cup \{0\}\}$ and
let $X_N^\ast$ be the set of words over $X_N$ including the empty word $\be$.
Denoted by $\fA_N$ the free noncommutative polynomial algebra in $X_N$, i.e., the algebra of words on $X_N^\ast$.
The \emph{weight} of a word $\bfw \in \fA_N$, denoted by $|\bfw|$,
is the number of letters contained in $\bfw$, and its \emph{depth}, denoted by $\dep(\bfw)$,
is the number of letters $x_\eta$ ($\eta\in\gG_N$) contained in $\bfw$.

Further, let $\fA_N^1$ denote the subalgebra of $\fA_N$ consisting of words not ending with $x_0$. Hence, $\fA_N^1$ is generated by words of the form $y_{m,\mu}:=x_0^{m-1}x_\mu$ for $m\in \mathbb N$ and $\mu \in \gG_N$.
Define the alphabet $Y_N=\{y_{k,\mu}: k\in \mathbb N,\mu \in \gG_N\}$ and $Y_N^\ast$ is
the set of words (including the empty word) over $Y_N$.  Additionally,
let $\fA_N^0$ denote the subalgebra of $\fA_N^1$ with words not beginning with $x_1$ and not ending with $x_0$.
The words in $\fA_N^0$ are called \emph{admissible words.}

We now equip $\fA_N^1$ with a Hopf algebra structure $(\fA_N^1,\ast,\widetilde{\Delta}_\ast)$.
The \emph{stuffle product} $\ast \colon \fA_N^1 \otimes \fA_N^1 \to \fA_N^1$ is defined as follows:
\begin{enumerate}[(ST1)]
 \item $\be \ast w := w \ast \be := w$,
 \item $y_{m,\mu}u \ast y_{n,\nu}v :=y_{m,\mu}(u \ast y_{n,\nu}v)+y_{n,\nu}(y_{m,\mu}u \ast v) + y_{m+n,\mu\nu}(u\ast v)$,
\end{enumerate}
for any word $u,v,w\in \fA_N^1$,  $m,n\in \mathbb N$ and $\mu,\nu\in \gG_N$. Then linearly extend it to $\fA_N^1$.
The coproduct $\widetilde{\Delta}_\ast \colon \fA_N^1 \to \fA_N^1\otimes \fA_N^1$ is defined by deconcatenation:
\begin{align*}
 \widetilde{\Delta}_\ast(y_{s_1,\eta_1}\cdots y_{s_d,\eta_d}):= \sum_{j=0}^d y_{s_1,\eta_1}\cdots y_{s_j,\eta_j}\otimes y_{s_{j+1},\eta_{j+1}} \cdots  y_{s_d,\eta_d}.
\end{align*}
Note that $(\fA_N^0,\ast)$ is a sub-algebra  (but not a sub-Hopf algebra).

We also need another Hopf algebra structure $(\fA_N,\sha, \widetilde{\Delta}_\sha)$ which
will provide the shuffle relations. Here, the \emph{shuffle product} $\sha\colon \fA_N\otimes \fA_N\to \fA_N$
is defined as follows:
\begin{enumerate}[(SH1)]
 \item $\be \sha w := w \sha \be := w$,
 \item $au \sha b v := a(u\sha bv) + b(au\sha v)$,
\end{enumerate}
for any word $u,v,w \in \fA_N$ and $a,b\in X_N$. Then linearly extend it to $\fA_N$.
The coproduct $\widetilde{\Delta}_\sha \colon \fA_N \to \fA_N\otimes \fA_N$ is
defined (again) by deconcatenation:
\begin{align*}
 \widetilde{\Delta}_\sha (x_{\eta_1}\cdots x_{\eta_d}):= \sum_{j=0}^d x_{\eta_1}\cdots x_{\eta_j}\otimes x_{\eta_{j+1}}\cdots  x_{\eta_{d}},
\end{align*}
where $\eta_1,\ldots,\eta_d\in \gG_N\cup \{0\}$.
Note that both $(\fA_N^1,\sha)$ and $(\fA_N^0,\sha)$ are sub-algebras
(but not as sub-Hopf algebras).

Finally, we remark that both  $(\fA_N^1,\ast)$ and $(\fA_N,\sha)$ are commutative
and associative algebras.

\section{Finite colored multiple zeta values}

We recall that finite MZVs and finite Euler sums are elements of the $\Q$-ring
\begin{align*}
\calA:=\prod_{p\in\calP} \F_p \bigg/\bigoplus_{p\in\calP} \F_p.
\end{align*}
Note that for $N=1,2$,  $\calA$ can be identified with $\calA(N)$
since all primes greater than 2 are odd and we can safely disregard the prime $p=2$.
By our choice of the primes in $\calP(N)$ it follows immediately from Fermat's Little Theorem
that the Frobenius endomorphism ($p$-power map at prime $p$-component)
is given by the \emph{conjugation} $(\ol{a_p})_p\in \calA(N)$, where
\begin{align}\label{equ:pPower=inv}
 \ol{\sum_{j=0}^{N-1} c_j \xi_N^j}:=\sum_{j=0}^{N-1} c_j \xi_N^{N-j}, \quad \quad (c_j\in \F_p).
\end{align}

The following lemma implies that $\calA(N)$ is in fact a $\Q(\gG_N)$-vector space.

\begin{lem}
The field $\Q(\gG_N)$ can be embedded into $\calA(N)$ diagonally.
\end{lem}

\begin{proof}
 The map $\phi\colon \mathbb Q \to \calA(N)$, $r\mapsto (\phi_p(r))_{p\in \calP(N)}$ given by $\phi(0)=(0)_p$ and
 \begin{align*}
  \phi_p(r):=\begin{cases}
              r \pmod p & \text{~if~} \ord_p(r)\geq 0, \\
              0 & \text{~otherwise},
             \end{cases}
 \end{align*}
embeds $\mathbb Q$ diagonally in $\calA(N)$ due to the fundamental theorem of arithmetic. The cyclotomic field $\mathbb Q(\gG_N)$ is given by
\begin{align*}
 \mathbb Q(\gG_N):=\left\{\sum_{j=0}^{N-1}a_j\xi_N^j: a_j\in \mathbb Q \right\},
\end{align*}
where $\xi_N\in\gG_N $ is a primitive element. Therefore $\varphi\colon \Q(\gG_N) \to \calA(N)$ defined by
\begin{align*}
 \sum_{j=0}^{N-1}a_j\xi_N^j \longmapsto \left(\sum_{j=0}^{N-1} \phi_p(a_j)\xi_N^j  \right)_{p\in \calP(N)}
\end{align*}
is an embedding.
\end{proof}

In this section, we study $\mathbb Q(\gG_N)$-linear relations among FCVs by developing a double shuffle picture.
First, similar to MZVs, the stuffle product is simply induced by the defining series of
FCVs \eqref{equ:defnFCV}. For example, we have
\begin{align*}
 \sum_{p>k>0}\frac{\alpha^k}{k^a} \sum_{p>l>01}\frac{\beta^l}{l^b}  = \sum_{p>k>l>0}\frac{\alpha^k\beta ^l}{k^al^b} +  \sum_{p>l>k>0}\frac{\alpha^k\beta ^l}{k^al^b} +  \sum_{p>k>0}\frac{(\alpha\beta)^k}{k^{a+b}}
\end{align*}
for $a,b\in \mathbb N$ and $\alpha,\beta \in \gG_N$.
On the other hand, the shuffle product is more involved than that for the MZVs
since apparently there is no integral representation for FCVs available. However, we
will deduce the shuffle relations by using the integral representation of the single variable
multiple polylogarithms.

Define the $\Q(\gG_N)$-linear map $\zeta_{\calA(N),\ast}\colon \fA_N^1\to\calA(N)$ by setting
\begin{align*}
 \zeta_{\calA(N),\ast}(w):=\zeta_{\calA(N)}\lrp{\bfs}{\bfeta}
\end{align*}
for any word $\displaystyle w=W\lrp{\bfs}{\bfeta}:=y_{s_1,\eta_1} \cdots y_{s_d,\eta_d}
\in \fA_N^1$, where $\bfs=(s_1,\ldots,s_d)\in \N^d$ and $\bfeta=(\eta_1,\dots,\eta_d)\in(\gG_N)^d$.

\begin{thm}\label{thm:stuffleFCV}
The map $\zeta_{\calA(N),\ast}\colon (\fA_N^1,\ast)\to\calA(N)$ is an algebra homomorphism.
\end{thm}

\begin{proof}
Let $u,v\in \fA_N^1$. It is easily seen by induction on $|u|+|v|$ that
\begin{align*}
\zeta_{\calA(N),\ast}(u\ast v)= \zeta_{\calA(N),\ast}(u)\zeta_{\calA(N),\ast}(v),
\end{align*}
which concludes the proof.
\end{proof}

We now define a map $\bfp\colon \fA_N^1\to \fA_N^1$ by setting
\begin{align*}
 \bfp(y_{s_1,\eta_1}y_{s_2,\eta_2}\cdots y_{s_d,\eta_d})= y_{s_1,\eta_1}y_{s_2,\eta_1\eta_2}\cdots y_{s_d,\eta_1\eta_2\cdots \eta_d}
\end{align*}
and its inverse $\bfq\colon \fA_N^1\to \fA_N^1$ by setting
\begin{align*}
 \bfq(y_{s_1,\eta_1}y_{s_2,\eta_2}\cdots y_{s_d,\eta_d})= y_{s_1,\eta_1}y_{s_2,\eta_2 \eta_1^{-1}}\cdots y_{s_d,\eta_d\eta_{d-1}^{-1}}.
\end{align*}
Further, set the map $\tau\colon \fA_1^1 \to \fA_1^1$ by defining $\tau(\be):=\be$ and
\begin{align*}
 \tau(x_0^{s_1-1}x_1\cdots x_0^{s_d-1}x_1) := (-1)^{s_1+\cdots+s_d} x_0^{s_d-1}x_1\cdots x_0^{s_1-1}x_1
\end{align*}
and extended to $\fA_1^1$ by linearity.

The next theorem can be proved in the same manner as that for \cite[Thm.~3.11]{Zhao15}.
In order to be self-contained, we present its complete proof. For any
$s,s_1,\dots,s_d\in\N$ and $\xi,\xi_1,\dots,\xi_d\in\gG_N$, we define
the function with complex variable $z$ with $|z|<1$ by
\begin{align*}
\zeta_\sha(y_{s,\xi} y_{s_1,\xi_1} \dots y_{s_d,\xi_d};z)
:=& \Li_{s,s_1,\dots,s_d}(z\xi,\xi_1/\xi,\xi_2/\xi_1,\dots,\xi_d/\xi_{d-1}) \\
=&\, \int_{[0,z]} \left(\frac{dt}{t}\right)^{s-1} \frac{dt}{\xi^{-1}-t}
\left(\frac{dt}{t}\right)^{s_1-1} \frac{dt}{\xi_1^{-1}-t} \cdots
\left(\frac{dt}{t}\right)^{s_d-1} \frac{dt}{\xi_d^{-1}-t},
\end{align*}
where $[0,z]$ is the straight line from 0 to $z$. It is clear that
$\zeta_\sha(-;z)$ is well-defined.
By the shuffle relation of iterated integrals, we get
\begin{align}\label{equ:zetaShaz}
\zeta_\sha(u;z)\zeta_\sha(v;z)=\zeta_\sha(u\sha v;z)
\end{align}
for all $u,v\in Y_N^\ast$.

\begin{thm}\label{thm:shuffleFCV}
Let $N\ge 1$. Define the map $\zeta_{\calA(N),\sha}:= \zeta_{\calA(N),\ast} \circ \bfq\colon \fA_N^1\to\calA(N)$.
Then we have
\begin{align*}
 \zeta_{\calA(N),\sha}(u \sha v)=\zeta_{\calA(N),\sha}(\tau(u)v)
\end{align*}
for any $u\in \fA_1^1$ and $v\in \fA_N^1$. These relations are called linear shuffle relations
for the FCVs.
\end{thm}

\begin{proof}
Obviously, it suffices to prove
\begin{align*}
 \zeta_{\calA(N),\sha}((x_0^{s-1}x_1 u) \sha v) =(-1)^s \zeta_{\calA(N),\sha}(u \sha (x_0^{s-1}x_1v))
\end{align*}
for $s\in \mathbb N$, $u\in \fA_1^1$ and $v\in \fA_N^1$. For simplicity we use the notation $a:=x_0$ and $b:=x_1$ in the rest of this proof.
Let $w$ be a word in $\fA_N^1$, i.e., there exist $\bfs \in \mathbb{N}^d$ and $\bfxi\in (\gG_N)^d$ such that $w=y_{s_1,\xi_1} \dots y_{s_d,\xi_d}$.
Then there exists $\bfeta\in (\gG_N)^d$, such that $\displaystyle \bfq(w)=W\lrp{\bfs}{\bfeta}$.
For any prime $p>2$, the coefficient of
$z^p$ in $\zeta_\sha(bw;z)$ is given by
\begin{equation*}
  {\rm Coeff}_{z^p}\Big[\zeta_\sha\lrp{1,s_1,\dots,s_d}{z,\xi_1,\dots,\xi_d} \Big]=
  \frac{1}{p}\sum_{p>k_1>\cdots>k_d>0}
    \frac{\eta_1^{k_1}\cdots \eta_d^{k_d}}{k_1^{s_1}\cdots k_d^{s_d}}
  =\frac{1}{p} H_{p-1}(\bfq(w)),
\end{equation*}
where
\begin{align*}
 H_k(y_{s_1,\eta_1} \dots y_{s_d,\eta_d}):=\sum_{k\geq k_1>\cdots>k_d>0}
    \frac{\eta_1^{k_1}\cdots \eta_d^{k_d}}{k_1^{s_1}\cdots k_d^{s_d}}.
\end{align*}
Observe that
\begin{align}\label{eq:shdecomp}
b \Big( (a^{s-1}b u) \sha v -(-1)^s   u \sha (a^{s-1}b v)\Big)
=\sum_{j=0}^{s-1} (-1)^{j} (a^{s-1-j}b u) \sha (a^j b v).
\end{align}
By applying $\zeta_\sha(-;z)$ to \eqref{eq:shdecomp} and then extracting
the coefficients of $z^p$  from both sides we obtain
\begin{align*}
&\, \frac{1}{p}\Big( H_{p-1}\circ\bfq\big((a^{s-1}b u) \sha v\big)
    -(-1)^s  H_{p-1}\circ\bfq\big(u \sha (a^{s-1}b v)\big)\Big)\\
=&\,\sum_{j=0}^{s-1} (-1)^{j} {\rm Coeff}_{z^p}
    \big[\zeta_\sha(a^{s-1-j}b u;z)\zeta_\sha(a^{j}b v;z)\big] \\
=&\,\sum_{j=0}^{s-1} (-1)^{j}\sum_{j=1}^{p-1}
{\rm Coeff}_{z^j}\big[\zeta_\sha(a^{s-1-j}b u;z) \big]
{\rm Coeff}_{z^{p-j}}\big[\zeta_\sha(a^{j}b v;z)\big]
\end{align*}
by \eqref{equ:zetaShaz}. Since $p-j<p$ and $j<p$ the last sum
is $p$-integral. Therefore we get
\begin{equation*}
H_{p-1}\circ\bfq((a^{s-1}b u) \sha v)\equiv
(-1)^s  H_{p-1}\circ\bfq(u \sha (a^{s-1}b v)) \pmod{p}
\end{equation*}
which completes the proof of the theorem.
\end{proof}


\section{Symmetrized colored multiple zeta values}
\subsection{Regularizations of colored MZVs}
Since the regularized colored MZVs (recalled in Theorem \ref{theo:Racinet} below)
are polynomials of $T$, both of the two kinds of SCVs defined by
\eqref{equ:astSCV} and \eqref{equ:shaSCV} are \emph{a priori} also polynomials of $T$.
We show first that they are in fact constant complex numbers. To this end,
we define two maps $\zeta_\ast,\zeta_\sha \colon \fA_N^0 \to \CC$ such that for any word
$\displaystyle w=W\lrp{\bfs}{\bfeta}\in \fA_N^0$
\begin{align*}
  \zeta_\ast(w):=\zeta\lrp{\bfs}{\bfeta} \quad \quad\text{~and~}\quad \quad \zeta_\sha(w):=\zeta_\sha \Lrp{s_1,s_2,\ldots,\ \, s_d\ \, }{\eta_1,\frac{\eta_{2}}{\eta_1},\ldots, \frac{\eta_{d}}{\eta_{d-1}}}.
\end{align*}

\begin{lem}[\cite{Arakawa04,Racinet02}]
The maps $\zeta_\ast \colon (\fA_N^0,\ast) \to \CC$ and
$\zeta_\sha \colon (\fA_N^0,\sha) \to \CC$ are algebra homomorphisms.
\end{lem}

This first preliminary result is well-known. The map $\zeta_\ast$ originates from
the series definition \eqref{def:CMZV} while the map $\zeta_\sha$ comes from
the integral representation by setting $z=1$ in \eqref{equ:zetaShaz}.

The following results of Racinet \cite{Racinet02} addressing the regularization of colored MZVs
generalize those for the MZVs first discovered by Ihara, Kaneko and Zagier \cite{Ihara06}.

\begin{thm}\label{theo:Racinet}
Let $N$ be a positive integer.
\begin{enumerate}[(i)]
  \item The algebra homomorphism $\zeta_\ast \colon (\fA_N^0,\ast) \to \CC$ can be extended to a homomorphism  $\zeta_{\ast}(-;T)\colon(\fA_N^1,\ast)\to\CC[T]$, where $\zeta_{\ast}(y_{1,1};T)=T$.
  \item The algebra homomorphism $\zeta_\sha \colon (\fA_N^0,\sha) \to \CC$ can be extended to a homomorphism  $\zeta_{\sha}(-;T)\colon(\fA_N^1,\sha)\to\CC[T]$, where $\zeta_{\sha}(x_1;T)=T$.
  \item For all $w\in\fA_N^1$, we have
  \begin{align*}
    \zeta_{\sha}\big(\bfp(w);T\big)=\rho \big(\zeta_{\ast}(w;T)\big),
  \end{align*}
  where $\rho\colon \CC[T]\to \CC[T]$ is a $\CC$-linear map such that
  \begin{align*}
    \rho(e^{Tu})=\exp\left(\sum_{n=2}^\infty  \frac{(-1)^n}{n}\zeta(n)u^n\right)e^{Tu}
  \end{align*}
  for all $|u|<1$.
\end{enumerate}
\end{thm}

In the above theorem, to signify the fact that the images of $\zeta_{\sha}$ and $\zeta_{\ast}$ are polynomials of $T$, we have used the notation $\zeta_{\sha}(w;T)$ and $\zeta_{\ast}(w;T)$.

\begin{prop}\label{prop:SCVconst}
For  all $\bfs\in \N^d$ and $\bfeta\in (\gG_N)^d$, both
$\displaystyle \zeta_\ast^\Sy\lrp{\bfs}{\bfeta}$
and $\displaystyle \zeta_\sha^\Sy\lrp{\bfs}{\bfeta}$ are constant complex values, i.e.,
they are independent of $T$.
\end{prop}

\begin{proof}
We start with the shuffle version
$\displaystyle \zeta_\sha^\Sy\lrp{s_1,\ldots, s_d}{\eta_1,\ldots,\eta_d}$.
If $(s_j,\eta_j)\ne (1,1)$ for all $j=1,\ldots,d$ then clearly it is finite.
In the case that $(s_j,\eta_j)=(1,1)$ for all $j=1,\ldots,d$ then the binomial theorem implies
\begin{equation*}
 \zeta_\sha^\Sy\lrp{\bfs}{\bfeta}=\sum_{j=0}^d (-1)^{d-j} \frac{T^j}{j!}\frac{T^{d-j}}{(d-j)!} =0.
\end{equation*}
Otherwise, we may assume for some $k$ and $l\ge 1$ we have $(s_k,\eta_k)\ne(1,1)$,
$(s_{k+1},\eta_{k+1})=\cdots=(s_{k+l},\eta_{k+l})=(1,1)$,
and $(s_{k+l+1},\eta_{k+l+1})\ne(1,1)$. We only need to show that
\begin{equation*}
  \sum_{j=0}^l
 (-1)^{j}  \zeta_\sha\lrpT{s_{k+j},\dots,s_1}{\eta_{k+j},\ldots,\eta_1} \zeta_\sha\lrpT{s_{k+j+1},\dots,s_d}{\eta_{k+j+1},\ldots,\eta_d}
\end{equation*}
is finite. Consider the words in $\fA_\sha^1$ corresponding to the
above values. We may assume
$x_\gl u=x_0^{s_k-1}x_{\eta_k} \cdots x_0^{s_1-1}x_{\eta_1}$ (which is
the empty word if $k=0$) and
$x_\mu v=x_0^{s_{k+l+1}-1}x_{\eta_{k+l+1}} \cdots x_0^{s_d-1}x_{\eta_d}$
(which is the empty word if $k+l=d$),
where $\gl,\mu\ne 1$. So we only need to show that
\begin{equation}\label{equ:symmtrizeWordsAdm}
\sum_{j=0}^l (-1)^j (x_1^jx_\gl u)\sha(x_1^{l-j} x_\mu v) \in \fA_N^0.
\end{equation}
Let $q_0=0$ and $q_m=1$ for all $m\ne 0$. Then we have
\begin{align*}
&\,\sum_{j=0}^l (-1)^j (x_1^jx_\gl u) \sha (x_1^{l-j} x_\mu v)\\
=&\, q_k x_\gl(u\sha x_1^l x_\mu v)
 +\sum_{j=1}^l (-1)^{j}x_1 (x_1^{j-1} x_\gl u\sha x_1^{l-j} x_\mu v)\\
+&\, q_{k+l-d}(-1)^l x_\mu(x_1^l x_\gl u \sha  v) +
\sum_{j=0}^{l-1} (-1)^j x_1 (x_1^j x_\gl u\sha x_1^{l-j-1}x_\mu v) \\
=&\,q_k x_\gl (u\sha x_1^l x_\mu v)+q_{k+l-d}(-1)^l x_\mu( x_1^l x_\gl u \sha v)\in \fA_N^0,
\end{align*}
where we have used the substitution $j\to j+1$ in the first sigma
summation which is canceled by the second sigma summation.

For the stuffle version $  \zeta_\ast^\Sy\lrp{\bfs}{\bfeta}$ we can use
the same idea because the stuffing parts, i.e., the contraction of two beginning letters, 
always produce admissible words. We leave the details to the interested reader.
\end{proof}

\subsection{Generating series}
Using the algebraic framework of Section \ref{sec:algframe}, we define for
$\bfs:=(s_1,\ldots,s_d)\in \N^d$ and $\bfeta:=(\eta_1,\ldots,\eta_d)\in (\gG_N)^d$
the following maps:
\begin{align*}
 \zeta_\ast^\Sy &\colon \fA_N^1 \to \CC, \quad \quad \zeta_\ast^\Sy(y_{s_1,\eta_1}\cdots y_{s_d,\eta_d}):=\zeta_\ast^\Sy\lrp{\bfs}{\bfeta}, \\
 \zeta_{\sha}^\Sy &\colon \fA_N^1 \to \CC, \quad \quad \zeta_\sha^\Sy(y_{s_1,\eta_1}\cdots y_{s_d,\eta_d}):=\zeta_\sha^\Sy\Lrp{s_1,s_2,\ldots,\ \, s_d\ \, }{\eta_1,\frac{\eta_{2}}{\eta_1},\ldots, \frac{\eta_{d}}{\eta_{d-1}}}.
\end{align*}

In order to study the SCVs effectively, we need to utilize their generating series, which
should be associated with some dual objects of the Hopf algebras
$(\fA_N^1,\ast,\widetilde{\Delta}_\ast)$ and $(\fA_N,\sha, \widetilde{\Delta}_\sha)$.
So we denote by $\hfA_N$ the completion of $\fA_N$ with respect to the weight
(and define $\hfA_N^1$ and $\hfA_N^0$ similarly).

Let $R$ be any $\Q$-algebra. Define the coproduct $\gD_\ast$ on
$R\langle\!\langle Y_N^\ast \rangle\!\rangle:=\hfA_N^1\otimes_\Q R$ by
$\gD_\ast(\be):=\be \otimes\be$ and
\begin{align*}
\gD_\ast(y_{k,\xi}):=\be \otimes y_{k,\xi}+y_{k,\xi}\otimes \be
    +\sum_{a+b=k,\,a,b\in\N \atop
     \gl\eta=\xi,\,\gl,\eta\in\gG_N} y_{a,\gl}\otimes y_{b,\eta}
\end{align*}
for all $k\in\N$ and $\xi\in\gG_N$. Then extend it $R$-linearly.
One can check easily that
$(R\langle\!\langle Y_N^\ast \rangle\!\rangle, \ast,\gD_\ast)$
is the dual to the Hopf algebra $(R\langle Y_N^\ast \rangle, \ast,\tilde{\gD}_\ast)$.

Take $R:=\CC[T]$. Let $\Psi_\ast^T$ be the generating series of
$\ast$-regularized colored MZVs, i.e.,
\begin{align*}
\Psi_\ast^T:=\sum_{w\in Y_N^\ast} \zeta_\ast(w;T) w \in \CC[T]\langle\!\langle  Y_N^\ast \rangle\!\rangle,
\end{align*}
which can be regarded as an element in the regular ring of functions
of $\CC[T]\langle Y_N^\ast\rangle$ by the above consideration.
Namely, we can define $\Psi_\ast^T[w]$ to be the coefficient of $w$ in $\Psi_\ast^T$.
Further we set $\Psi_\ast:=\Psi_\ast^0$.

The shuffle version of $\Psi_\ast$ is more involved even though the basic idea is the same.
First, for any $\Q$-algebra $R$ we can define the coproduct $\gD_\sha$ on
$R\langle\!\langle X_N^\ast \rangle\!\rangle:=\hfA_N\otimes_\Q R$ by
\begin{align*}
 \gD_\sha(\be):=\be \otimes \be \quad \text{and} \quad \gD_\sha(x_\gl):=x_\gl\otimes\be+\be \otimes x_\gl
\end{align*}
for all $\gl\in\gG_N\cup\{0\}$. Let $\eps_\sha$ be the counit such that
$\eps_\sha(\be)=1$ and $\eps_\sha(w)=0$ for all $w\ne\be$.
Then the Hopf algebra $(R\langle\!\langle X_N^\ast \rangle\!\rangle,\sha, \gD_\sha)$
is dual to $(R\langle X_N^\ast \rangle,\sha,\tilde{\gD}_\sha)$.
Now we can define
\begin{align*}
\Psi_\sha^T:=\sum_{w\in Y_N^\ast} \zeta_{\sha}(\bfp(w);T) w \in \CC[T]\langle\!\langle Y_N^\ast \rangle\!\rangle,
\end{align*}

\begin{lem}\label{lem:group-likePsi}
We have
\begin{align*}
 \Psi_\ast^T=\exp(T y_{1,1}) \Psi_\ast \quad \text{and} \quad
 \Psi_\sha^T=\exp(T y_{1,1}) \Psi_\sha.
\end{align*}
\end{lem}

\begin{proof}
Set $\widehat{\Psi}_\sha(T):=\sum_{w\in Y_N^\ast} \zeta_{\sha}(w;T) w$.
Then $\widehat{\Psi}_\sha(T)$ is group-like for $\gD_\sha$. Further, $\Psi_\ast^T$ is group-like
for $\gD_\ast$. It follows from Cor.~2.4.4 and Cor.~2.4.5 of \cite{Racinet02} that
\begin{align*}
 \Psi_\ast^T=\exp(T y_{1,1}) \Psi_\ast \quad \text{and} \quad
 \widehat{\Psi}_\sha(T)=\exp(T y_{1,1}) \widehat{\Psi}_\sha(0),
\end{align*}
since $x_1=y_{1,1}$.
Since $\bfq(y_{1,1}^n w)=y_{1,1}^n \bfq(w)$ for all $n\in\Z_{\ge 0}$ and $w\in Y_N^\ast$, applying $\bfq$ to the second equality leads to
\begin{align*}
  \sum_{w\in Y_N^\ast} \zeta_{\sha}(w;T) \bfq(w)
  =\exp(T y_{1,1})\bigg( \sum_{w\in Y_N^\ast} \zeta_{\sha}(w;0) \bfq(w) \bigg),
\end{align*}
which is equivalent to $\Psi_\sha^T=\exp(T y_{1,1}) \Psi_\sha$, as desired.
\end{proof}

For all $s_1,\dots,s_d\in\N$ and $\eta_1,\dots,\eta_d\in\gG_N$, we define the
anti-automorphism $\inv\colon Y_N^\ast \to Y_N^\ast$ by
\begin{equation*}
\inv(y_{s_1,\eta_1} \cdots y_{s_d,\eta_d}):=
(-1)^{s_1+\cdots+s_d}\, \ol{\eta_1}\cdots \ol{\eta_d}\, y_{s_d,\eta_d}\cdots y_{s_1,\eta_1} .
\end{equation*}
and then extend it to $\CC\langle\!\langle Y_N^\ast \rangle\!\rangle$ by linearity.

\begin{lem}\label{lem:invPsi}
We have
\begin{align*}
\inv(\Psi_\ast)\Psi_\ast=\inv(\Psi_\ast^T)\Psi_\ast^T=&\,\sum_{w\in Y_N^\ast} \zeta_\ast^\Sy(w) w, \\
\inv(\Psi_\sha)\Psi_\sha=\inv(\Psi_\sha^T)\Psi_\sha^T=&\,\sum_{w\in Y_N^\ast} \zeta_\sha^\Sy(\bfp(w)) w.
\end{align*}
\end{lem}

\begin{proof}
In each of two lines above, the second equality follows directly from the definition
while the first is an immediate consequence of  Proposition~\ref{prop:SCVconst}.
\end{proof}

\begin{thm}\label{theo:moduloz2}
For any $\bfs\in\N^d,\bfeta\in(\gG_N)^d$, we have
\begin{align*}
 \zeta_\sha^\Sy\lrp{\bfs}{\bfeta}\equiv  \zeta_\ast^\Sy\lrp{\bfs}{\bfeta} \pmod{\zeta(2)}.
\end{align*}
\end{thm}

\begin{proof}
By Lemma~\ref{lem:group-likePsi} and Theorem~\ref{theo:Racinet}, we get
\begin{align*}
\exp(T y_{1,1}) \Psi_\sha= \Psi_\sha^T=\rho(\Psi_\ast^T)=\exp(T y_{1,1}) \gL(y_{1,1})\Psi_\ast,
\end{align*}
where $\gL(y_{1,1}):=\exp\left(\sum_{n=2}^\infty\frac{(-1)^n}{n}\zeta(n) y_{1,1}^n\right)$.
Therefore $\Psi_\sha= \gL(y_{1,1})\Psi_\ast$.
Using the fact $\zeta(2n)\in \zeta(2)^n\mathbb{Q}$ for $n\in \mathbb{N}$ implies
\begin{align*}
\inv(\Psi_\sha)\Psi_\sha=\inv(\Psi_\ast)\gL(-y_{1,1})\gL(y_{1,1})\Psi_\ast
\equiv \inv(\Psi_\ast)\Psi_\ast \pmod{\zeta(2)}.
\end{align*}
Hence, the theorem follows from Lemma~\ref{lem:invPsi}.
\end{proof}


\subsection{Shuffle and stuffle relations}
We first prove the stuffle relations of the SCVs.
\begin{thm}\label{thm:astSCVmorphism}
 The map $\zeta_\ast^\Sy\colon (\fA_N^1,\ast) \to \CC$ is a homomorphism of algebras, i.e.
 \begin{align*}
  \zeta_\ast^\Sy(w\ast w') = \zeta_\ast^\Sy(w)\zeta_\ast^\Sy(w')
 \end{align*}
for all $w,w'\in \fA_N^1$.
\end{thm}

\begin{proof}
Since $\zeta_\ast\colon (\fA_N^1,\ast)\to \CC[T]$ is an algebra homomorphism, its generating series $\Psi_\ast^T$ must be a group-like element of $\gD_*$, i.e., $\gD_*(\Psi_\ast^T) = \Psi_\ast^T \otimes \Psi_\ast^T$. Further, it can be checked in a straight-forward manner that $\Delta_\ast \circ\inv = (\inv \otimes \inv)\circ \Delta_\ast$. Thus we get
 \begin{align*}
  \Delta_\ast\big(\inv(\Psi_\ast^T)\Psi_\ast^T\big) = \big(\inv(\Psi_\ast^T)\Psi_\ast^T\big) \otimes \big(\inv(\Psi_\ast^T)\Psi_\ast^T\big)
 \end{align*}
and Lemma \ref{lem:invPsi} implies the claim.
 \end{proof}

For the shuffle relations we need the \emph{generalized Drinfeld associator
$\Phi=\Phi_N$ at level $N$}.
Enriquez \cite{Enriquez2007} defined it as the renormalized holonomy from 0 to 1 of
\begin{equation} \label{equ:generalizedDrinfeld}
H'(z)  = \left(  \sum_{\eta\in \gG_N \cup \{0\}}
\frac{x_\eta}{z-\eta} \right) H(z),
\end{equation}
i.e., $\Phi:=H_1^{-1}H_0$, where $H_0,H_1$ are the solutions
of \eqref{equ:generalizedDrinfeld} on the open interval $(0,1)$
such that $H_0(z) \sim z^{x_0}=\exp(x_0 \log z)$ when $z\to 0^+$,
$H_1(z) \sim (1-z)^{x_1}=\exp(x_1 \log(1-z))$ when $z\to 1^-$.

\begin{thm} 
The generalized Drinfeld associator $\Phi$ is the unique element in the Hopf algebra $(\CC\langle\!\langle X_N^\ast \rangle\!\rangle,\sha,\gD_\sha,\eps_\sha)$ such that
\begin{itemize}
  \item[\upshape{(i)}] $\Phi$ is group-like, i.e., $\eps_\sha(\Phi)=1$ and $\gD_\sha(\Phi)=\Phi\otimes\Phi$,
  \item[\upshape{(ii)}]  $\Phi[x_0]=\Phi[x_1]=0$,
  \item[\upshape{(iii)}]  $\displaystyle \Phi[\bfp(x_0^{s_1-1} x_{\eta_1}\dots  x_0^{s_d-1} x_{\eta_d})]
  =(-1)^d \zeta \lrpTZ{\bfs}{\bfeta}$
  for any $\bfs:=(s_1,\dots,s_d)\in\N^d$ and $\bfeta:=(\eta_1,\dots,\eta_d)\in(\gG_N)^d$.
\end{itemize}
\end{thm}
\begin{proof}
The uniqueness, the statements in (i), (ii) and the case $(s_1,\eta_1)\ne(1,1)$ of (iii)
of the theorem follow directly from \cite[App.]{Enriquez2007} and \cite[Prop. 5.17]{Deligne05}.
By Theorem~\ref{theo:Racinet} (ii), if $(s_1,\eta_1)=(1,1)$ then $\zeta \lrpTZ{\bfs}{\bfeta}$
is determined uniquely by the admissible values from the shuffle structure by using (ii). But $\Phi[\bfp(x_0^{s_1-1} x_{\eta_1}\dots  x_0^{s_d-1} x_{\eta_d})]$ is also determined uniquely
by the coefficients of admissible words from the same
shuffle structure so that (iii) still holds even if $(s_1,\eta_1)=(1,1)$. This completes
the proof of the theorem.
\end{proof}

For any $\eta\in\gG_N$, we define the map $r_\eta\colon X_N^*\to X_N^*$ by
setting
\begin{align*}
r_\eta(x_0^{a_1} x_{\eta_1}^{b_1}\dots x_0^{a_d} x_{\eta_d}^{b_d})
:=x_0^{a_1} x_{\eta_1/\eta}^{b_1}\dots x_0^{a_d} x_{\eta_d/\eta}^{b_d}
\end{align*}
for all $a_1,b_1,\dots,a_d,b_d\in\Z_{\ge 0}$ and $\eta_1,\dots,\eta_d\in\gG_N$.

\begin{lem}\label{lem:eta-twistPhi}
For any $\eta\in\gG_N$, the \emph{$\eta$-twist $\Phi_{\eta}$ of $\Phi$} defined by
\begin{align*}
   \Phi_{\eta}:=\sum_{w\in X_N^\ast} \Phi[r_\eta(w)] w
\end{align*}
is group-like for $\gD_\sha$ and $\Phi_{\eta}^{-1}$ is well-defined.
\end{lem}
\begin{proof}
For any words $u,v\in X_N^*$, we have
\begin{multline*}
\gD_\sha(\Phi_\eta) [u\otimes v]=\Phi_\eta[ u\sha  v]=\Phi[r_\eta ( u\sha v)]\\
=\Phi[r_\eta (u) \sha r_\eta (v)]
=\Phi[r_\eta (u)]\Phi[r_\eta(v)]
=(\Phi_\eta\otimes\Phi_\eta)[u\otimes v],
\end{multline*}
since $\Phi$ is group-like. Thus $\gD_\sha(\Phi_\eta)=\Phi_\eta\otimes\Phi_\eta$.
Further, $\Phi_{\eta}^{-1}$ is well-defined since
\begin{align*}
   (\Phi_{\eta})^{-1}=\sum_{w\in X_N^\ast} (-1)^{|w|}\Phi[r_\eta(w)] \revs{w}
   =\sum_{w\in X_N^\ast} (-1)^{|w|}\Phi[r_\eta(\revs{w})]w=(\Phi^{-1})_{\eta},
\end{align*}
where $\revs{w}=\ga_d\ga_{d-1}\cdots\ga_1$ is the reversal of the word
$w=\ga_1\cdots \ga_{d-1}\ga_d\in X_N^\ast$ with the letters $\ga_j\in X_N$ for $j=1,\ldots,d$.
\end{proof}

\begin{thm}\label{thm:SCVDrinfeldAss}
For any word $w\in \fA_N^1$, we have
\begin{align*}
 \zeta_\sha^\Sy(w)=(-1)^d\sum_{\eta\in\gG_N} \ol{\eta}\, \Phi_{\eta}^{-1} x_{\eta} \Phi_{\eta} [x_1 w].
\end{align*}
\end{thm}

\begin{proof}
 First we observe that
 \begin{align}\label{eq:invphi}
  \Phi_{\eta}^{-1}[x_1 x_0^{s_1-1} x_{\eta_1}\cdots x_{\eta_{j-1}} x_0^{s_j-1}]
  = (-1)^{s_1+\cdots+s_j}\Phi_{\eta}[ x_0^{s_j-1} x_{\eta_{j-1}}\cdots x_{\eta_1} x_0^{s_1-1}x_1].
 \end{align}
 for $j=1,\ldots,d$. Then we obtain (by setting $\eta_0:=1$)
 \begin{align*}
    & (-1)^d\sum_{\eta\in\gG_N} \ol{\eta}\, \Phi_{\eta}^{-1} x_{\eta} \Phi_{\eta} \big[x_1 x_0^{s_1-1} x_{\eta_1}\dots x_0^{s_d-1} x_{\eta_d}\big] \\
  = &~ (-1)^d \sum_{j=0}^d \ol{\eta_j}\, \Phi_{\eta_j}^{-1}\big[x_1 x_0^{s_1-1} x_{\eta_1}\cdots x_{\eta_{j-1}}x_0^{s_{j}}\big]  \Phi_{\eta_j}\big[x_0^{s_{j+1}-1}x_{\eta_{j+1}}\cdots x_0^{s_d-1}x_{\eta_d}\big]\\
  = &~ (-1)^d \sum_{j=0}^d (-1)^{s_1+\cdots+s_j}\ol{\eta_j}\, \Phi_{\eta_j}\big[x_0^{s_{j}}x_{\eta_{j-1}}\cdots x_0^{s_1-1} x_1\big] \Phi_{\eta_j}\big[x_0^{s_{j+1}-1}x_{\eta_{j+1}}\cdots x_0^{s_d-1}x_{\eta_d}\big]\\
  = &~ (-1)^d \sum_{j=0}^d (-1)^{s_1+\cdots+s_j}\ol{\eta_j}\, \Phi\big[x_0^{s_{j}}x_{\eta_{j-1}/\eta_j}\cdots x_0^{s_1-1} x_{\eta_0/\eta_j}\big] \Phi\big[x_0^{s_{j+1}-1}x_{\eta_{j+1}/\eta_j}\cdots x_0^{s_d-1}x_{\eta_d/\eta_j}\big]\\
  = &~ (-1)^d \sum_{j=0}^d (-1)^{s_1+\cdots+s_j}\ol{\eta_j}\,
   \cdot (-1)^j \zeta_\sha\LrpTZ{\ s_j\ , \ldots,s_1}
  {\frac{\eta_{j-1}}{\eta_{j}},\ldots,\frac{\eta_0}{\eta_1} }
  \cdot (-1)^{d-j} \zeta_\sha\LrpTZ{s_{j+1},\ldots,\ s_d\ }
  {\frac{\eta_{j+1}}{\eta_j},\ldots,\frac{\eta_d}{\eta_{d-1}} }\\
  = & ~ \zeta_\sha^\Sy \Lrp{s_1,s_2,\ldots,\ \, s_d\ \, }{\eta_1,\frac{\eta_{2}}{\eta_1},\ldots, \frac{\eta_{d}}{\eta_{d-1}}},
 \end{align*}
by Proposition~\ref{prop:SCVconst}. We have completed our proof.
\end{proof}

\begin{thm}\label{thm:shuffleSCV}
Let $N\ge 1$. For any $w,u\in\fA_1^1$ and $v\in\fA_N^1$, we have
\begin{align*}
 \zeta_\sha^\Sy(u\sha v)= \zeta_\sha^\Sy(\tau(u)v).
\end{align*}
These relations are called linear shuffle relations for the SCVs.
\end{thm}

\begin{proof}
It suffices to prove
\begin{align*}
 \zeta_\sha^\Sy( (x_0^{s-1}x_1 u) \sha v) =(-1)^s \zeta_\sha^\Sy(u \sha (x_0^{s-1}x_1 v)).
\end{align*}
for all $s\in\N$. We observe that
\begin{equation}\label{equ:sumOfShuffles}
x_1 \Big( (x_0^{s-1}x_1 u)\sha v-(-1)^s u\sha(x_0^{s-1} x_1 v)\Big)
=\sum_{i=0}^{s-1} (-1)^i (x_0^{s-1-i}x_1 u)\sha(x_0^{i}x_1v).
\end{equation}
By Theorem.~\ref{thm:SCVDrinfeldAss}, it suffices to show that the image of \eqref{equ:sumOfShuffles} under $E:=\Phi_{\eta}^{-1}x_1\Phi_{\eta}$ vanishes
for all $\eta\in\gG_N$. By Lemma~\ref{lem:eta-twistPhi}
\begin{equation*}
\gD_\sha(E)=(\Phi_{\eta}^{-1}\ot\Phi_{\eta}^{-1})
(x_1\ot \be+\be\ot x_1)(\Phi_{\eta}\ot\Phi_{\eta})
=E\ot \be+\be\ot E.
\end{equation*}
Therefore $E$ is a primitive element for $\gD_\sha$ so that
we can regard it as a Lie element, namely, it acts on shuffle products like a derivation.
Hence, for any nonempty words $u,v\in X_N^*$,
\begin{equation*}
E[u \sha v]=E[u]\epsilon_\sha[v]+ \epsilon_\sha[u]E[v]=0.
\end{equation*}
This completes the proof by Theorem~\ref{thm:SCVDrinfeldAss}
since none of the factors in the shuffle products on the
right-hand side of Eq.~\eqref{equ:sumOfShuffles} is the
empty word as the letter $x_1$ appears in every factor.
\end{proof}


\section{Reversal relations of FCVs and SCVs}

One of the simplest but very important relations among FCVs and SCVs
are the following reversal relations.
\begin{prop}\label{prop:reversalFCVSCV}
Let $\bfs \in \N^d$, $\bfeta \in (\gG_N)^d$,
and define $\pr(\bfeta):=\prod_{j=1}^d \eta_j$. Then we have
\begin{align}\label{equ:reversalFCV}
\zeta_{\calA(N)} \lrp{\revs{\bfs}}{\revs{\bfeta}}
=&\,(-1)^{|\bfs|} \pr(\ol{\bfeta}) \zeta_{\calA(N)} \lrp{\bfs}{\ol{\bfeta}}, \\
\zeta_\sha^\Sy\lrp{\revs{\bfs}}{\revs{\bfeta}}
=&\,(-1)^{|\bfs|} \pr(\ol{\bfeta}) \zeta_\sha^\Sy\lrp{\bfs}{\ol{\bfeta}},\label{equ:reversalSCVsha}\\
\zeta_\ast^\Sy\lrp{\revs{\bfs}}{\revs{\bfeta}}
=&\,(-1)^{|\bfs|} \pr(\ol{\bfeta}) \zeta_\ast^\Sy\lrp{\bfs}{\ol{\bfeta}},\label{equ:reversalSCVast}
\end{align}
where $\revs{\bfa}:=(a_d,\dots,a_1)$ is the reversal of $\bfa:=(a_1,\ldots,a_d)$ and $\ol{\bfa}:=(\ol{a_1},\dots,\ol{a_d})$ is the componentwise conjugation of $\bfa$.
\end{prop}
\begin{proof}
Equation \eqref{equ:reversalFCV} follows easily from the substitution $k_j\to p-k_j$
for the indices in \eqref{equ:defnFCV} by the condition $p\equiv -1\pmod{N}$.
Equations~\eqref{equ:reversalSCVsha}
and \eqref{equ:reversalSCVast} follow easily from the definitions.
\end{proof}


\section{Numerical examples}\label{sec:numex}
In this last section, we provide some numerical examples in support of
Conjecture~\ref{conj:Main}.
We will need some results from the level 1 case.
\begin{prop}\label{prop:homogeneousWols1} \emph{(\cite[Theo.~2.13]{Zhao08b})}
Let $s$, $d$ and $N$ be positive integers. Then
\begin{align}\label{equ:homogeneousWols}
\zeta_{\calA(N)}\Lrp{\{s\}^d}{\{1\}^d}=0.
\end{align}
\end{prop}

\begin{eg}\label{eg:level3}
At level 3, by Proposition~\ref{prop:reversalFCVSCV} and
Proposition~\ref{prop:homogeneousWols1}, for all $w\in \N$, we have 
\begin{alignat*}{3}
     \zeta_{\calA(3)}\lrp{w}{1}=&\, 0, \quad &\zeta_{\calA(3)}\lrp{w,w}{1,1}=&\, 0, \quad &  \zeta_{\calA(3)}\lrp{w}{\xi_3}=&\, \xi_3^2 (-1)^w \zeta_{\calA(3)}\lrp{w}{\xi_3^2},\\
     \zeta_\ast^\Sy\lrp{w}{1}=&\, \gd_w\zeta(w), \quad &\zeta_\ast^\Sy\lrp{w,w}{1,1}=&\,
     \gd_w\zeta(w)^2-\zeta(2w), \quad &   \zeta_\ast^\Sy\lrp{w}{\xi_3}=&\, \xi_3^2(-1)^w \zeta_\ast^\Sy\lrp{w}{\xi_3^2},
\end{alignat*}
where $\gd_w= (1+(-1)^w)$,  and
\begin{alignat*}{3}
     \zeta_{\calA(3)}\lrp{w,w}{1,\xi_3}=&\, \xi_3^2  \zeta_{\calA(3)}\lrp{w,w}{\xi_3^2,1}, \quad &
     \zeta_\ast^\Sy\lrp{w,w}{1,\xi_3}=&\, \xi_3^2 \zeta_\ast^\Sy\lrp{w,w}{\xi_3^2,1}, \\   \zeta_{\calA(3)}\lrp{w,w}{1,\xi_3^2}=&\,\xi_3\zeta_{\calA(3)}\lrp{w,w}{\xi_3,1},\quad &  \zeta_\ast^\Sy\lrp{w,w}{1,\xi_3^2}=&\,\xi_3 \zeta_\ast^\Sy\lrp{w,w}{\xi_3,1}, \\     \zeta_{\calA(3)}\lrp{w,w}{\xi_3,\xi_3}=&\,\xi_3\zeta_{\calA(3)}\lrp{w,w}{\xi_3^2,\xi_3^2},\quad &
      \zeta_\ast^\Sy\lrp{w,w}{\xi_3,\xi_3}=&\,\xi_3 \zeta_\ast^\Sy\lrp{w,w}{\xi_3^2,\xi_3^2}. 
\end{alignat*} 
\end{eg}

\begin{eg}\label{eg:level4} 
At level 4, by \cite[Cor.\ 2.3]{Tauraso10}, we have
\begin{align}
     \zeta_{\calA(4)}\lrp{w}{1}=0, \quad
     \zeta_{\calA(4)}\lrp{w}{i^2}=
\left\{
  \begin{array}{ll}
    0, & \hbox{if $w$ is even;} \\
-2q_2 , & \hbox{if $w=1$;} \\
 -2(1-2^{1-w}) \gb_w, & \hbox{otherwise,}
  \end{array}
\right.  \notag\\
     \zeta_{\calA(4)}\lrp{w}{i}=\sum_{p>k>0} \frac{i^k}{k^w} =E^{(w)}_p+iO^{(w)}_p,\quad
     \zeta_{\calA(4)}\lrp{w}{i^3}=\sum_{p>k>0} \frac{i^{3k}}{k^w}=E^{(w)}_p-iO^{(w)}_p,\label{equ:depth1N=4}
\end{align}
where $q_2:=((2^{p-1}-1)/p )_{p\in\calA(4)}$ is the $\calA(4)$-Fermat quotient, $\gb_w:=(B_{p-w}/w)_{p\in\calA(4),p>w}$ is the  $\calA(4)$-Bernoulli number and
\begin{align*}
E^{(w)}_p:=\sum_{p>2k>0}\frac{(-1)^k}{(2k)^w},\quad \quad O^{(w)}_p:=\sum_{p>2k+1>0}\frac{(-1)^k}{(2k+1)^w}.
\end{align*}  
To find more relations, first we observe that
\begin{align*}
\sum_{k=(p-1)/2}^{p-1} \frac{(-1)^k} k
\equiv \sum_{p-k=(p-1)/2}^{p-1} \frac{(-1)^{p-k}}{p-k}
\equiv  \sum_{p>2k>0}\frac{(-1)^k}k\equiv \frac12E^{(1)}_p  \pmod{p}.
\end{align*}
By Eq.~\eqref{equ:depth1N=4}, we have
\begin{align*}
\zeta_{\calA(4)}\lrp{1}{i^2}=2 \zeta_{\calA(4)}\lrp{1}{i}+2 \zeta_{\calA(4)}\lrp{1}{i^3}.
\end{align*}
This is consistent with Conjecture~\ref{conj:Main} since 
\begin{align}\label{equ:piInSCV14}
 \zeta_\ast^\Sy\lrp{1}{i^2}=
 2\zeta_\ast^\Sy\lrp{1}{i}+2\zeta_\ast^\Sy\lrp{1}{i^3}+2\pi
 \equiv
 2\zeta_\ast^\Sy\lrp{1}{i}+2\zeta_\ast^\Sy\lrp{1}{i^3} \pmod{2\pi i \Q(i)}.
\end{align} 
Indeed, we have
\begin{align*}
 \zeta_\ast^\Sy\lrp{1}{i^2}=&\, \Li_1(i^2)-\ol{i^2} \Li_1(\ol{i^2})=2\Li_1(-1)=-2\log(2),\\
 \zeta_\ast^\Sy\lrp{1}{i}=&\, \Li_1(i)+i \Li_1(i^3)=-\log(1-i)-i\log(1+i)
 =-\frac12\log 2+\frac{\pi i}{2}-\frac{i}{2}\log 2+\frac{\pi}{2},\\
 \zeta_\ast^\Sy\lrp{1}{i^3}=&\, \Li_1(i^3)-i \Li_1(i)=-\log(1+i)+i\log(1-i)
 =-\frac12\log 2-\frac{\pi i}{2}+\frac{i}{2}\log 2+\frac{\pi}{2}.
\end{align*}
\end{eg}

{}From numerical evidence, we form the following conjecture:
\begin{conj}\label{conj:level4FCVbasis}
Let $w\ge 1$. When the level $N=3,4$, the $\Q(\xi_N)$-vector space $\FCV_{w,N}$ has
the following basis:
\begin{align}\label{equ:BasisN=3,4}
    \left\{ \zeta_{\calA(N)}\Lrp{\{1\}^w}{\xi_N,\xi_N^{\gd_2},\dots,\xi_N^{\gd_w}}\colon \gd_2,\dots,\gd_{w}\in \{0,1\} \right\}.
\end{align}
\end{conj}


To find as many $\Q(\gG_N)$-linear relations as possible in weight $w$
we may choose all the known relations in weight $k<w$, multiply
them by $\displaystyle \zeta_{\calA(N)}\lrp{\bfs}{\bfeta}$ for all $\bfs$ of weight $w-k$
and all $\bfeta$, and then expand all the products
using the stuffle relation proved in Theorem \ref{thm:stuffleFCV}. All the $\Q(\gG_N)$-linear
relations among FCVs of the same weight produced in this way are
called \emph{linear stuffle relations of FCVs}.
We can similarly define linear stuffle relations of SCVs

By using  linear  shuffle and stuffle relations and the reversal relations we can show that
Eq.~\eqref{equ:BasisN=3,4} in Conjecture~\ref{conj:level4FCVbasis} are generating sets
in the cases $1\le w\le 4$ and $N=3,4$ but cannot show their linear independence at the moment.

Concerning Conjecture~\ref{conj:Main} (i), the inclusion $\nSCV_{w,N}\subseteq \CMZV_{w,N}$ is
trivial but the opposite inclusion seems difficult. Note that since $2\pi i\in \SCV_{1,4}$ by
Eq.~\eqref{equ:piInSCV14} we have $\nSCV_{w,4}=\SCV_{w,4}$ for all $w$
by the stuffle relations. But from Example~\ref{eg:level3} we see that $\SCV_{1,3}$
is generated by
\begin{align*}
\zeta_\ast^\Sy\lrp{1}{\xi_3}=\zeta\lrp{1}{\xi_3}-\xi_3^2\zeta\lrp{1}{\xi_3^2}
\equiv (1-\xi_3^2)\zeta\lrp{1}{\xi_3} \pmod{2\pi i \Q(\xi_3)},
\end{align*}
which should imply that
$\SCV_{1,3}=\Big\langle \zeta\lrp{1}{\xi_3} \Big\rangle \ne
\nSCV_{1,3}$, since conjecturally $\zeta\lrp{1}{\xi_3}=(\pi i-3\log 3)/6$
and $\pi$ are algebraically independent.

Moreover, when the weight $w\le 3$ and the level $N=3,4$ we have numerically
verified in both spaces of Conjecture~\ref{conj:Main} (ii), exactly the same
linear relations leading to the dimension upper bound $2^{w-1}$ hold
(with error bounded by $10^{-99}$ for SCVs
and with congruence checked for all primes $p<312$ and $p=1019$ in $\calP(N)$ for FCVs).

\begin{eg}\label{eg:higherWtlevel3,4}
In weight 2 level 3 and 4, we can prove rigorously that
\begin{align*}
3\zeta_{\calA(3)}\lrp{2}{\xi_3}=&\, 2\zeta_{\calA(3)}\lrp{1,1}{\xi_3,\xi_3}(1-\xi_3)
-6\zeta_{\calA(3)}\lrp{1,1}{\xi_3,1},\\
3\zeta_\ast^\Sy\lrp{2}{\xi_3}=&\, 2\zeta_\ast^\Sy\lrp{1,1}{\xi_3,\xi_3}(1-\xi_3)
-6\zeta_\ast^\Sy\lrp{1,1}{\xi_3,1}+\frac{(2\pi i)^2}{12}(2+\xi_3),\\
\zeta_{\calA(4)}\lrp{2}{i}=&\, (i-1)\zeta_{\calA(4)}\lrp{1,1}{i,1}-i\zeta_{\calA(4)}\lrp{1,1}{i,i}, \\
\zeta_\ast^\Sy\lrp{2}{i}=&\, (i-1)\zeta_\ast^\Sy\lrp{1,1}{i,1}-i\zeta_\ast^\Sy\lrp{1,1}{i,i}
+ \frac{2\pi i}{12} \Big( 2(1+i)\zeta_\ast^\Sy\lrp{1}{i}-i\zeta_\ast^\Sy\lrp{1}{-1} \Big).
\end{align*}
In weight 3 level $N=3$ and 4, we have verified numerically
\begin{align*}
\zeta_{\calA(3)}\lrp{1,1,1}{1,\xi_3^2,\xi_3}=&\,3\xi_3\zeta_{\calA(3)}\lrp{1,1,1}{\xi_3,1,1},\\
\zeta_\ast^\Sy\lrp{1,1,1}{1,\xi_3^2,\xi_3}=&\,3\xi_3\zeta_\ast^\Sy\lrp{1,1,1}{\xi_3,1,1}+
\frac{(2\pi i)^2}{24}(1-\xi_3)\zeta_\ast^\Sy\lrp{1}{\xi_3}-\frac{(2\pi i)^3}{144}(1-\xi_3),
\end{align*}
and
\begin{align*}
(15-66i)\zeta_{\calA(4)}\lrp{1,2}{1,1}=&\, 48\Big(
(1+i)\zeta_{\calA(4)}\lrp{1,1,1}{i,i,1}
-(1+2i)\zeta_{\calA(4)}\lrp{1,1,1}{i,i,i} \\
+&\,2(1-i)\zeta_{\calA(4)}\lrp{1,1,1}{i,1,i}
-3(1-i)\zeta_{\calA(4)}\lrp{1,1,1}{i,1,1} \Big),
\end{align*}
while
\begin{multline*}
(15-66i)\zeta_\ast^\Sy\lrp{1,2}{1,1}=48\Big(
(1+i)\zeta_\ast^\Sy\lrp{1,1,1}{i,i,1}
-(1+2i)\zeta_\ast^\Sy\lrp{1,1,1}{i,i,i}\\
+2(1-i)\zeta_\ast^\Sy\lrp{1,1,1}{i,1,i}
-3(1-i)\zeta_\ast^\Sy\lrp{1,1,1}{i,1,1}\Big)\\
+2\pi i \Big[
(106+2i)\zeta_\ast^\Sy\lrp{1,1}{i,i}
-(2-88i) \zeta_\ast^\Sy\lrp{1,1}{i,1}
-(\frac{3}2-11i)\zeta_\ast^\Sy\lrp{1,1}{-1,-1}
-(64+26i)\zeta_\ast^\Sy\lrp{1,1}{i,-1} \Big].
\end{multline*}
\end{eg}

We end this paper by an intriguing mystery. During our Maple computation, we found
that FCVs should have an interesting structure over $\Q$. For example, numerical
evidence suggests that all the weight $w$ and level 4 FCVs should generate a
dimension $2^w$ vector space over $\Q$.  We wonder how to relate this $\Q$ structure
to that of the CMZVs at level 4.

\bibliographystyle{alpha}
\bibliography{library}

\end{document}